\documentclass[11pt, letterpaper]{article}
\usepackage{setspace}
\usepackage{amsfonts}
\usepackage{amsmath}
\usepackage{amsthm}
\usepackage{amssymb}
\usepackage{mathrsfs}
\usepackage{mathtools}
\usepackage[hyphens]{url}
\usepackage{hyperref}
\usepackage[x11names]{xcolor}
\usepackage{graphicx}
\usepackage{xspace}
\usepackage{listings}
\usepackage{inconsolata}
\usepackage{units}
\usepackage{epigraph}
\usepackage{soul}
\usepackage[capitalize,noabbrev,nameinlink]{cleveref}
\usepackage[margin=1in]{geometry}
\usepackage{microtype}

\newcommand{\mypar}[1]{\noindent\textbf{#1}}

\lstdefinestyle{mylistingstyle}{
	language=Python,
    backgroundcolor=\color{Ivory1},   
    basicstyle=\ttfamily,
    breakatwhitespace=false,         
    breaklines=true,                 
    captionpos=b,
    frame=single,            
    keepspaces=true,                 
    numbers=left,	
    showspaces=false,                
    showstringspaces=false,
    showtabs=false,                  
    tabsize=2,
    keywordstyle=\bfseries,
    morekeywords={None}
}

\newtheorem{theorem}{Theorem}[]
\newtheorem{lemma}[theorem]{Lemma}

\DeclareMathOperator{\ord}{ord}

\DeclareMathOperator{\modulo}{mod}

\title{Efficient Lifting of Discrete Logarithms Modulo Prime Powers}

\makeatletter
\def\sharedfootnote{\footnotemark[\value{footnote}]\addtocounter{footnote}{-1}}
\makeatother

\author{
Giovanni Viglietta\thanks{University of Aizu, \texttt{\{vigliett, kachi\}@u-aizu.ac.jp}} \and
Yasuyuki Kachi\sharedfootnote
}

\date{}
\begin{document}

\maketitle

\begin{abstract}
We present a deterministic algorithm that, given a prime $p$ and a solution $x \in \mathbb Z$ to the discrete logarithm problem $a^x \equiv b \pmod p$ with $p\nmid a$, efficiently lifts it to a solution modulo $p^k$, i.e., $a^x \equiv b \pmod {p^k}$, for any fixed $k \geq 1$.

The algorithm performs $k(\lceil \log_2 p\rceil +2)+O(\log p)$ multiplications modulo $p^k$ in the worst case, improving upon prior lifting methods by at least a factor of 8.
\end{abstract}

\section{Introduction}\label{s:1}
\mypar{Summary.} In 1984, Bach~\cite{Bach84} gave an efficient deterministic algorithm for \emph{lifting} a discrete logarithm modulo a prime $p$ to a higher prime-power modulus $p^k$. In this paper, we improve on Bach's result by providing a more efficient algorithm for the same task. In \cref{s:2} we describe the algorithm itself, while in \cref{s:3} we prove its correctness, and in \cref{s:4} we analyze its performance and compare it to Bach's algorithm's. \cref{s:5} contains some concluding remarks.

\medskip
\mypar{Significance.} Our result has applications in computational number theory, where arithmetic modulo $p^k$ arises naturally in algorithms involving $p$-adic methods, local fields, and root-lifting techniques such as Hensel's lemma.

Many standard discrete logarithm algorithms, such as index calculus and its variants~\cite{Adleman1979}, are designed for finite fields and do not extend naturally to the ring~$\mathbb{Z}_{p^k}$. While generic methods like baby-step giant-step~\cite{Shanks1969} or Pollard's rho~\cite{Pollard1978} can still be applied, they do not leverage any known solution modulo~$p$ and are slower than methods that exploit the structure of~$\mathbb{Z}_{p^k}$. On the other hand, general-purpose algorithms such as Pohlig--Hellman~\cite{PohligHellman1978} rely on factoring the multiplicative group order $\varphi(p^k) = p^{k-1}(p - 1)$, which is impractical when $p$ is large and $p - 1$ is difficult to factor.

Our contribution is also relevant in certain cryptographic settings, including fault and side-channel attack models~\cite{Boneh1997,Coron1999,Kocher1999}, where computations modulo $p^k$ are used to track error propagation through arithmetic operations. It further applies to lattice-based cryptographic constructions such as NTRU and Ring-LWE~\cite{Albrecht2016,Hoffstein1998,Peikert2016}, where modular arithmetic over rings modulo a prime power plays a central role in both algorithm design and implementation.

\medskip
\mypar{Notation.} We adopt standard number-theoretic notation and terminology. $\mathbb Z$ indicates the set of all integers; $\mathbb Z^+$ denotes the positive integers, and $\mathbb N$ the non-negative integers. If $b\in \mathbb Z$ is an integer multiple of $a\in \mathbb Z$, we write $a\mid b$; otherwise, we write $a\nmid b$. The \emph{congruence} $a\equiv b\pmod m$ means that $m\mid a-b$. For a prime $p\in\mathbb N$, $\nu_p(a)$ indicates the \emph{$p$-adic valuation} of $a$, i.e., the maximum $r\in\mathbb N$ such that $p^r\mid a$; by convention, $\nu_p(0)=+\infty$. If $a$ and $m$ have no common prime divisor, they are said to be \emph{relatively prime}. If this is the case, then $\ord_m(a)$ denotes the \emph{multiplicative order} of $a$ modulo $m$, that is, the smallest positive integer $r$ such that $a^r\equiv 1\pmod m$. We will often make use of \emph{Fermat's Little Theorem}: If $p$ is a prime and $p\nmid a$, then $a^{p-1}\equiv 1\pmod p$. Equivalently, $\ord_{p}(a)\mid p-1$. We adopt the convention that $0^0=1$.

\section{The Algorithm}\label{s:2}
\mypar{Language.} Our algorithm is presented in \cref{alg:1}. Python~3.0 was chosen because its syntax is close to pseudocode and it natively supports arithmetic on arbitrarily large integers.\footnote{We remark that, in Python~3.0, the operator \texttt{a//b} returns $\lfloor a/b\rfloor$, the operator \texttt{a\%b} returns the remainder after dividing $a$ by $b$ (it is always non-negative if $b>0$), \texttt{pow(a,n)} computes $a^n$, and \texttt{pow(a,n,m)} computes $a^n$ modulo $m$.} That said, the algorithm can be straightforwardly implemented in any language that supports multiprecision integer arithmetic, such as Java (via the \texttt{BigInteger} class), C++ (using libraries such as GMP or NTL), or Mathematica.

\begin{lstlisting}[style=mylistingstyle,caption={Python 3.0 implementation of our algorithm. It assumes that $k\in \mathbb N$, $p$ is a prime, $p\nmid a$, and $a^z\equiv b \pmod{p}$. It returns $x\in \mathbb N$ such that $a^x \equiv b \pmod{p^k}$, or \texttt{None} if no such $x$ exists. \label{alg:1}},captionpos=t,float,belowcaptionskip=\bigskipamount,mathescape=true]
def discrete_log_lift(a, z, b, p, k):
    m = pow(p, k)
    z = z % (p - 1)
    a = a % m
    b = b % m
    c = pow(a, z, m)
    y = 0
    if p == 2:
        u = 2
        e = pow(a, 2, m)
    else:
        u = 1
        e = pow(a, p - 1, m)
    h = e - 1
    r = 0
    while h % p == 0 and r < k:
        h = h // p
        r = r + 1
    v = pow(p, r)
    if (b - c) % v != 0:
        if p > 2 or (b - a) % v != 0:
            return None
        c = a
        y = 1
    g = pow(h * b, p - 2, p)
    while b != c:
        d = (g * ((b - c) // v)) % p
        y = y + d * u
        if r + r >= k:
            f = (1 + (e - 1) * d) % m
            e = (1 + (e - 1) * p) % m
        else:
            f = pow(e, d, m)
            e = (f * pow(e, p - d, m)) % m
        c = (c * f) % m
        u = u * p
        v = v * p
        r = r + 1
    return (p - 1) * y + z
\end{lstlisting}

\medskip
\mypar{Overview.} The algorithm assumes that $a^z\equiv b\pmod p$ and returns $x$ such that $a^x\equiv b\pmod{p^k}$ (or reports that no such $x$ exists), provided that $p\nmid a$.\footnote{The case $p\mid a$ is of little interest, as the minimum solution $x$ can easily be found without knowledge of $z$, either via an exhaustive search, since necessarily $0\leq x\leq k$, or by solving the linear equation $\nu_p(a)x= \nu_p(b)$, after ruling out some trivial cases. The details are left to the reader.}

The final solution $x$ is expressed in the form $x=(p-1)y+z$ (Line~39), where $z$ is given, and $y$ is computed incrementally, one digit in base $p$ at a time. Specifically, the algorithm computes a sequence $(y_j)_{j\geq 0}$ such that $y_j<p^j$, that is, $y_j$ has $j$ digits in base $p$. In particular, $y_0=0$ (Line~7) and $y_{j+1} = y_j + d_jp^{j}$ with $0\leq d_j<p$ (Line~28). That is, $y_{j+1}$ is obtained by appending a new digit $d_j$ to the left of the base-$p$ expression of $y_j$.

Accordingly, we can define $x_j = (p-1)y_j+z$ and $c_j=a^{x_j} \modulo p^k$ (Lines~6, 29--35). As we will prove in \cref{s:3}, if the digit $d_j$ is chosen appropriately (Lines~25 and~27), then $c_{j+1}$ is a ``better approximation'' of $b$ than $c_j$, in some sense. Indeed, we can immediately verify that
$$c_0\equiv a^{x_0}\equiv a^z\equiv b \pmod p$$
and therefore $c_0$ and $b$ coincide at least on their first digit in base $p$. Then, as the computation progresses, each new $c_j$ matches $b$ on one additional digit. More specifically, if we define
$$r=\nu_p\left(a^{p-1}-1\right)$$
(Lines~13--18), then
$$c_j\equiv b\pmod{p^{r+j}}$$
Thus, by the time $j=k-r$, the numbers $c_j$ and $b$ are equal modulo $p^k$, and $x_j$ is the final solution.

\medskip
\mypar{Algorithmic details.} We will now flesh out our algorithm, providing additional details and discussing the cases where the solution $x$ does not exist. For convenience, we will make the assumption that $p$ is an odd prime; the special case $p=2$ is slightly different and will be described later. The proof of correctness is in \cref{s:3} and a performance analysis is in \cref{s:4}.

The algorithm performs preliminary computations up to Line~25, and the main loop starts at Line~26. The goal of the main loop is to compute the digits of $y$ until $a^{(p-1)y+z}\equiv b\pmod{p^k}$.
\begin{itemize}
\item {\bf Line~1:} The input to the algorithm are the integers $a$, $z$, $b$, $p$, $k$ such that $p$ is a prime, $p\nmid a$, $a^z\equiv b\pmod p$, and $k\geq 0$. The output will be an integer $x\geq 0$ such that $a^x\equiv b\pmod{p^k}$.
\item {\bf Line~2:} We compute $m=p^k$ as the modulus for our subsequent operations.
\item {\bf Lines~3, 4:} We ensure that $0\leq a,b<p^k$. This is a harmless normalization step, since we are working modulo $p^k$.
\item {\bf Line~5:} We ensure that $0\leq z<p-1$. Again, this normalization step is possible because, due to Fermat's Little Theorem, $\ord_p(a)\mid p-1$.
\item {\bf Lines~6, 7:} We set $y_0=0$ and $c_0=a^{x_0} \modulo p^k$, where (implicitly) $x_0=(p-1)y_0+z=z$.
\item {\bf Line~12:} We set $u_0=1$. This is simply a helper variable that is used to set the next digit of $y$ (Line~28) and is multiplied by $p$ at every iteration of the main loop (Line~36).
\item {\bf Line~13:} We set $e_0 = a^{p-1}\modulo p^k$. This is another helper variable that is used to update $c$ (Lines~33 and~35) and is raised to the $p$th power at every iteration of the main loop (Line~34).
\item {\bf Lines~14--18:} We set $r=\min\{\nu_p\left(a^{p-1}-1\right),k\}$ and $h=\frac{a^{p-1}-1}{p^r}$. Note that $h$ is an integer and $p\nmid h$, by definition of $\nu_p$. Bounding $r$ with $k$ serves the only purpose of preventing an infinite loop if $a^{p-1}\equiv 1\pmod{p^k}$.
\item {\bf Line~19:} We set $v_0 = p^r$. This helper variable is used to select the next digit of $c$ to be modified to match $b$ (Line~27), and is multiplied by $p$ at every iteration of the main loop (Line~37).
\item {\bf Lines~20--22:} Since the next digit of $c$ to be modified will be the $(r+1)$th, then the $(r+2)$th, etc., it follows that the first $r$ digits of $c$ will never change. Thus, if $c_0\not\equiv b\pmod{p^r}$, there can be no $j$ such that $c_j\equiv b\pmod{p^k}$. In this case, we can immediately report that the problem has no solution (this step will be fully justified in \cref{s:3}). Otherwise, since the first $r$ digits of $c$ are correct and all subsequent digits can be modified at will, the problem has a solution, which will be computed in the main loop.
\item {\bf Line~25--28:} We preliminarily compute $g\equiv (hb)^{p-2}\pmod p$. Note that, due to Fermat's Little Theorem, this is tantamount to computing $(hb)^{-1}$, provided that $p\nmid hb$. This number $g$ is used in the main loop to compute each new digit of $y$. Specifically, at the $(j+1)$th iteration of the loop, with $j\geq 0$, we extract the $(r+j+1)$th digit of $b$ and the $(r+j+1)$th digit of $c_{j}$, and we multiply their difference by $g \pmod p$ (Line~27). The result is $d_{j}$, which is the new $(j+1)$th digit of $y$. Thus, $y_{j+1} = y_{j} + d_{j}p^{j}$ (Line~28). This choice for $d_{j}$ guarantees that $c_{j+1}=a^{(p-1)y_{j+1}+z}$ matches $b$ on the first $r+j+1$ digits (this claim will be fully justified in \cref{s:3}).\footnote{This operation on $c$ does not only alter its $(r+j+1)$th digit, but possibly also its higher-order digits. However, these digits will be inspected and corrected in the next iterations of the main loop.}
\item {\bf Lines~29--31, 38:} This is an optimization that roughly cuts in half the number of exponentiations performed by the algorithm. Since this step is not essential, we can ignore it for now and postpone its discussion to \cref{s:4}.
\item {\bf Lines~33--35:} We compute $c_{j+1}\equiv a^{(p-1)y_{j+1}+z}\equiv c_{j}\cdot a^{(p-1)d_{j}p^{j}} \equiv c_{j}\cdot e_{j}^{d_{j}}\pmod{p^k}$. This is done by computing $f_{j}=e_{j}^{d_{j}}\modulo p^k$. Also, $e_{j+1}=e_{j}^{p}\modulo p^k$. As we will argue in \cref{s:4}, \emph{this is the bottleneck of the algorithm}. In fact, as an optimization, we compute $e_{j+1}$ as $f_{j}\cdot e_{j}^{p-d_{j}}\modulo p^k$, to reduce the exponent. Further optimizations of this critical step will be discussed in \cref{s:4}.
\item {\bf Lines~36, 37:} We compute $u_{j+1}=pu_{j}$ (which will be used to modify the next digit of $y$) and $v_{j+1}=pv_{j}$ (which will be used to select the next digit of $c$ and $b$).
\item {\bf Line~39:} When $b$ and $c_j$ are equal modulo $p^k$ (which happens after at most $k-r$ iterations of the main loop) the algorithm terminates and returns $x=x_j=(p-1)y_j+z$, which is guaranteed to satisfy $a^x\equiv a^{(p-1)y_j+z}\equiv c_j\equiv b\pmod{p^k}$, by construction.
\end{itemize}

\mypar{Special case $p=2$.}
The algorithm as described above only works if $p$ is an odd prime. If $p=2$, we necessarily have $z=0$, and therefore the equation $x_j=(p-1)y_j+z$ simplifies to $x_j=y_j$. As it turns out, however, our general setup fails when applied to computing the first binary digit of $y$; this point will be clarified in \cref{s:3}.\footnote{Incidentally, this is related to the well-known fact that there are primitive roots modulo $p^k$ for every odd prime $p$, but there are no primitive roots modulo $2^k$ for $k\geq 3$. That is, the standard proof by induction that there are multiplicative generators modulo $p^k$, due to Gauss, fails in the case $p=2$ for precisely the same reason why our \cref{l:2} cannot be generalized to $p=2$.} Thus, the first digit of $y$ has to be set ``manually'', depending on the first $r$ digits of $b$.

Fortunately, things start behaving as expected with the second binary digit of $y$. That is, adjusting the $(j+2)$th digit of $y$ (as opposed to the $(j+1)$th), with $j\geq 0$, allows us to modify the $(r+j+1)$th digit of $c=a^y$ at will (with the side effect of possibly altering its higher-order digits, as well) while leaving the previous digits of $c$ unchanged.

So, the main loop of the algorithm for the case $p=2$ is exactly the same as the $p>2$ version; the only differences are in the preliminary computations, as follows.
\begin{itemize}
\item {\bf Lines~8--10:} Since the first binary digit of $y$ will be chosen ``manually'', the main loop should start with its second digit. Accordingly, $u_0=p^1=2$ and $e_0\equiv a^{(p-1)p^1}\equiv a^2 \pmod{p^k}$. Note that, with these definitions, $r=\min\{\nu_2\left(a^2-1\right),k\}$ and $h=\frac{a^2-1}{2^r}$ (Lines~14--18).
\item {\bf Lines~20--24:} We choose the first binary digit of $y$, or we report that both choices are incompatible with $b$. Specifically, if $y_0=d_0=0$, we have $c_0=a^0=1$; and if $y_0=d_0=1$, we have $c_0=a^1=a$. If neither of these numbers matches the first $r$ digits of $b$, we report that the problem has no solution (Line~22). Otherwise, if $y_0=0$ is correct we simply proceed with the main loop (Line~20); if $y_0=1$ is correct we update the variables $y$ and $c$ accordingly before proceeding (Lines~23, 24). These steps will be fully justified in \cref{s:3}.
\end{itemize}

\section{Proof of Correctness}\label{s:3}
In this section we set out to justify all the claims we made in \cref{s:2} about the correctness of our algorithm. We first introduce three elementary lemmas. Note that \cref{l:1} is a well-known fact in number theory, but its proof has been included for self-containment.

\begin{lemma}\label{l:1}
Let $p$ be a prime, $j\in \mathbb Z^+$, and $x \in \mathbb{Z}$ such that $p\nmid x$. Then
$$\ord_{p^{j+1}}(x) \in \left\{ \ord_{p^j}(x),\ p \cdot \ord_{p^j}(x) \right\}$$
\end{lemma}
\begin{proof}
Let $d_j=\ord_{p^j}(x)$ for $j\in \mathbb Z^+$. By definition of $d_{j+1}$, we have 
$$x^{d_{j+1}}\equiv 1 \pmod{p^{j+1}}$$
and in particular
$$x^{d_{j+1}}\equiv 1 \pmod{p^{j}}$$
which implies that $d_j\mid d_{j+1}$. Now, since
$$x^{d_{j}}\equiv 1 \pmod{p^{j}}$$
there exists an integer $m$ such that $x^{d_j} = 1 + mp^j$. Thus,
$$x^{pd_j}= \left(x^{d_j}\right)^p= \left(1+mp^j\right)^p= \sum_{i=0}^p{p\choose i}m^ip^{ij}= 1 +{p\choose 1}mp^{j} + np^{2j}= 1 +mp^{j+1} + np^{2j}$$
for some integer $n$. Since $j\geq 1$, the last term in the previous sum vanishes modulo $p^{j+1}$. Hence,
$$x^{pd_j}\equiv 1 +mp^{j+1} + np^{2j} \equiv 1 \pmod{p^{j+1}}$$
We conclude that $d_j\mid d_{j+1}\mid pd_j$, and by the primality of $p$ either $d_{j+1}=d_j$ or $d_{j+1}=pd_j$.
\end{proof}

\begin{lemma}\label{l:2}
Let $p$ be an odd prime, $x \in \mathbb{Z}$, and $r\in \mathbb N$ such that $1\leq r\leq \nu_p(x-1)$. Then, for any $j\in \mathbb N$,
\begin{equation}\label{eq:a1}
x^{p^j} \equiv 1 + hp^{r+j} \pmod{p^{r+j+1}}
\end{equation}
where $h=\frac{x-1}{p^r}\mod p$.
\end{lemma}
\begin{proof}
We will prove \cref{eq:a1} by induction on $j\geq 0$. For $j=0$, \cref{eq:a1} becomes
$$x \equiv 1 + hp^{r} \pmod{p^{r+1}}$$
which holds by definition of $h$. Assuming \cref{eq:a1} to hold for some $j\geq 0$, let us prove it for $j+1$ in lieu of $j$. By assumption, $x^{p^j}=1+hp^{r+j}+mp^{r+j+1}$ for some integer $m$. Thus,
\begin{equation}\label{eq:a2}
x^{p^{j+1}} = \left(1+hp^{r+j}+mp^{r+j+1}\right)^p = \left(1+p^{r+j}(h+mp)\right)^p = \sum_{i=0}^p {p\choose i}p^{i(r+j)}(h+mp)^i
\end{equation}
Observe that the term ${p\choose i}p^{i(r+j)}$ is a multiple of $p^{r+j+2}$ for all $1< i< p$. Indeed,
$${p\choose i}=\frac{p(p-1)\cdots(p-i+1)}{i(i-1)\cdots 2\cdot 1}$$
is a multiple of $p$, because the prime factor $p$ in the numerator is not canceled by any factor in the denominator, since $i<p$. Therefore,
$$\nu_p\left({p\choose i} p^{i(r+j)} \right) = i(r+j)+1\geq 2(r+j)+1\geq r+j+2$$
which holds because $r\geq 1$ and $j\geq 0$. Moreover, for $i=p$, since $p$ is an odd prime, we have
$$\nu_p\left({p\choose p}p^{p(r+j)}\right) = p(r+j) \geq 3(r+j) \geq r+j+2$$
Thus, every term of the summation in \cref{eq:a2} vanishes modulo $p^{r+j+2}$ except possibly for $i=0$ and $i=1$. It follows that \cref{eq:a1} is verified for $j+1$:
$$x^{p^{j+1}} \equiv 1 + {p\choose 1}p^{r+j}(h+mp) \equiv 1 + hp^{r+j+1} + mp^{r+j+2} \equiv 1 + hp^{r+j+1} \pmod{p^{r+j+2}}$$
which concludes our proof by induction.
\end{proof}

\begin{lemma}\label{l:3}
Let $x \in \mathbb{Z}$ and $r\in \mathbb N$ such that $1\leq r\leq \nu_2(x^2-1)$. Then, for any $j\in \mathbb N$,
\begin{equation}\label{eq:b1}
x^{2^{j+1}} \equiv 1 + h2^{r+j} \pmod{2^{r+j+1}}
\end{equation}
where $h=\frac{x^2-1}{2^r}\mod 2$.
\end{lemma}
\begin{proof}
We will prove \cref{eq:b1} by induction on $j\geq 0$. For $j=0$, \cref{eq:b1} becomes
$$x^2 \equiv 1 + h2^{r} \pmod{2^{r+1}}$$
which holds by definition of $h$. Assuming \cref{eq:b1} to hold for some $j\geq 0$, let us prove it for $j+1$ in lieu of $j$. By assumption, $x^{2^{j+1}}=1+h2^{r+j}+m2^{r+j+1}$ for some integer $m$. Thus,
$$
x^{2^{j+2}} = \left(1+h2^{r+j}+m2^{r+j+1}\right)^2 = \left(1+2^{r+j}(h+2m)\right)^2 = 1 + 2^{r+j+1}(h+2m) + 2^{2(r+j)}(h+2m)^2
$$
Observe that the last term in the sum vanishes modulo $2^{r+j+2}$, provided that $r+j\geq 2$. If this is the case, we obtain \cref{eq:b1} for $j+1$, as desired:
$$
x^{2^{j+2}} \equiv 1 + 2^{r+j+1}(h+2m) \equiv 1 + h2^{r+j+1}+m2^{r+j+2} \equiv 1 + h2^{r+j+1} \pmod{2^{r+j+2}}
$$

Since $r\geq 1$, the only case where $r+j\geq 2$ does not hold is when $j=0$ and $r=1$. In this special case, \cref{eq:b1} becomes
$$x^2\equiv 1+2h \pmod 4$$
which holds by assumption for some $h\in \{0,1\}$. Thus $x^2$ is odd, implying that $x$ is also odd. If we write $x=1+2m$, we obtain
$$1+2h\equiv x^2 \equiv (1+2m)^2 \equiv 1+4m + 4m^2 \equiv 1+4m(m+1)\equiv 1 \pmod 4$$
and hence $h=0$. Moreover, since either $m$ or $m+1$ is an even number, we have
$$x^2\equiv 1+ 4m(m+1)\equiv 1\pmod 8$$
We conclude that \cref{eq:b1} is true for $j+1$, as well (recall that $j=0$, $r=1$, and $h=0$):
$$x^{2^{j+2}}\equiv x^{2^2}\equiv \left(x^2\right)^2 \equiv 1^2 \equiv 1 \equiv 1 + h2^{r+j+1} \pmod{2^{r+j+2}=8}$$
Thus, \cref{eq:b1} is verified for $j+1$ in every case, which concludes our proof by induction.
\end{proof}


We are finally ready to prove that our algorithm is correct.

\begin{theorem}\label{t:1}
The algorithm in \cref{alg:1} is correct. That is, if $p$ is a prime, $p\nmid a$, $a^z\equiv b\pmod p$, and $k\geq 0$, it outputs an integer $x\geq 0$ such that $a^x\equiv b\pmod{p^k}$ or \texttt{None} if no such $x$ exists.
\end{theorem}
\begin{proof}
We can immediately rule out the uninteresting special case $k=0$, because any equation is satisfied modulo $p^0=1$, and the algorithm correctly returns $x=z$ in this case. So, let $k> 0$ and let us assume that $p$ is an odd prime; the special case $p=2$ will be discussed later. A solution $x\geq 0$ can always be expressed in the form $x=(p-1)y+n$ with $y,n\in \mathbb N$ and $n<p-1$, by Euclidean division. We will first prove that returning \texttt{None} in Line~22 is correct.

Since $p\nmid a$, by Fermat's Little Theorem we have $a^{p-1}\equiv 1\pmod p$. Thus,
$$a^x\equiv a^{(p-1)y+n}\equiv \left(a^{p-1}\right)^{y}a^{n}\equiv 1^{y}a^{n}\equiv a^{n}\pmod p$$
Since $a^z\equiv b\pmod p$ by assumption, and $x$ must satisfy $a^x\equiv a^n\equiv b\pmod p$, we must have
$$n\equiv z\pmod{\ord_p(a)}$$
Define $r=\min\{\nu_{p}(a^{p-1}-1),k\}$ and $h=\frac{a^{p-1}-1}{p^r}$, as in \cref{alg:1}. Due to Fermat's Little Theorem, and because $k\geq 1$, we have $r\geq 1$ as well. Also, by definition of $\nu_{p}$, $h$ is an integer. Since $a^{p-1}\equiv 1\pmod{p^r}$, we have $\ord_{p^r}(a)\mid p-1$, and therefore $\ord_{p^r}(a)\leq p-1$.

\cref{l:1} with $x:=a$ and $j:=1,2,\dots, r-1$ implies that $\ord_{p^r}(a)=\ord_p(a)p^i$ with $0\leq i<r$. Since $\ord_{p^r}(a)\leq p-1$, we must have $i=0$ and hence $\ord_{p^r}(a)=\ord_p(a)$. Therefore,
$$n\equiv z\pmod{\ord_{p^r}(a)}$$
and so $a^{n}\equiv a^z\pmod{p^r}$. Now, if $x$ is a solution to $a^x\equiv b\pmod{p^r}$, then necessarily
$$b\equiv a^x\equiv a^{(p-1)y+n}\equiv \left(a^{p-1}\right)^{y}a^{n}\equiv a^{n}\equiv a^z\pmod{p^r}$$
So, if $a^z\not\equiv b\pmod{p^r}$, then there is no solution to the equation modulo $p^r$, and therefore no solution modulo $p^k$, because $r\leq k$ by construction. We conclude that returning \texttt{None} in Line~22, after checking that $c_0\equiv a^z\not\equiv b\pmod{p^r}$ (Line~20) is correct, as there is no solution $x$ in this case.

We will now prove that, if the condition $a^{z}\equiv b\pmod{p^r}$ holds (and so the algorithm does not return \texttt{None}), then there is a unique integer $y$, with $0\leq y<p^{k-r}$, that satisfies the equation $a^{(p-1)y+z}\equiv b \pmod{p^k}$. Moreover, we will show that the $k-r$ digits of $y$ in base $p$ are correctly computed in Lines~25, 27, and~28 of the algorithm, and this will conclude the proof of correctness for any odd prime $p$.

We will prove our claim by induction. Assume that the first $j\geq 0$ digits of $y$ have been correctly computed, i.e., we have $y_j<p^j$ such that
$$c_j\equiv a^{(p-1)y_j+z}\equiv b\pmod{p^{r+j}}$$
There are $p$ possible choices for the next digit of $y$; choosing a particular digit $d$, with $0\leq d<p$, gives rise to $y^{(d)}_{j+1} = dp^j+y_j$ and
$$c^{(d)}_{j+1}=a^{(p-1)y^{(d)}_{j+1} + z} = a^{(p-1)\left( dp^j + y_j \right) + z}$$
Applying \cref{l:2} with $x:=a^{p-1}$ (recall that $r\geq 1$), we obtain
$$a^{(p-1)p^j}\equiv 1+hp^{r+j}\pmod{p^{r+j+1}}$$
Thus, we can compute $c^{(d+1)}_{j+1}-c^{(d)}_{j+1}$, for $0\leq d<p-1$, as follows:
\begin{align*}
c^{(d+1)}_{j+1}-c^{(d)}_{j+1}
&= a^{(p-1)\left( (d+1)p^j + y_j \right) + z}
 - a^{(p-1)\left( dp^j + y_j \right) + z} \\
&= a^{(p-1)\left( dp^j + y_j \right) + z}
        \left( a^{(p-1)p^j} - 1 \right) \\
&= \left( a^{dp^j} \right)^{p-1}
        a^{(p-1)y_j + z}
        \left( a^{(p-1)p^j} - 1 \right) \\
&\equiv (1+m_1p)(b+m_2p)hp^{r+j} \\
&\equiv h b p^{r + j} \pmod{p^{r + j + 1}}
\end{align*}
for some integers $m_1$ and $m_2$. In other words, incrementing the $(j+1)$th digit of $y$ by one unit has the effect of increasing the $(r+j+1)$th digit of $c$ by $hb$ modulo $p$ (as well as, possibly, some higher-order digits of $c$).

Now, in order to match the $(r+j+1)$th digit of $c$ with the $(r+j+1)$th digit of $b$, and considering that $c_{j+1}^{(0)}=c_j$, we have to solve the following equation for the unknown $d$:
$$dhb\equiv \frac{b-c_j}{p^{r+j}}\pmod p$$
We will argue that $p\nmid hb$. We already know that $p\nmid a$ and $a^z\equiv b\pmod p$, and thus $p\nmid b$. Moreover, recall that $r=\min\{\nu_{p}(a^{p-1}-1),k\}$. If $r=k$, we already have $c_0\equiv b\pmod{p^k}$ before entering the main loop, and hence the algorithm immediately terminates with $x=z$, which is correct. Otherwise, if $r<k$, we must have $r=\nu_{p}(a^{p-1}-1)$, and therefore $p\nmid h=\frac{a^{p-1}-1}{p^r}$, by definition of $\nu_p$. We have proved that $p\nmid b$ and $p\nmid h$, and hence $p\nmid hb$. So, $hb$ is invertible modulo $p$, and we can write
$$d\equiv (hb)^{-1}\: \frac{b-c_j}{p^{r+j}}\equiv (hb)^{p-2}\: \frac{b-c_j}{p^{r+j}}\pmod p$$
due to Fermat's Little Theorem. So, there is exactly one solution $d$ modulo $p$, and that has to be the next digit of $y$. Accordingly, Lines~25, 27, and 28 of the algorithm update $y$ in the same way. This concludes our proof by induction that the algorithm correctly computes all digits of $y$ until $c=b$.

We will now turn to the special case $p=2$. The proof of correctness is essentially the same, with a few minor differences. The main reason why $p=2$ is special is that the proof of \cref{l:2} requires that $p\geq 3$. For $p=2$, we have to use \cref{l:3} instead.

Note that in this case $z=0$, and therefore we have to construct the final solution $x$ from scratch. Let us write $x=2m+y_0$, with $y_0\in \{0,1\}$, and recall that $r=\min\{\nu_2\left(a^2-1\right),k\}$ and $h=\frac{a^2-1}{2^r}$. Thus, if $x$ is a solution, we must have
$$b\equiv a^x\equiv a^{2m+y_0}\equiv a^{y_0}\equiv c_0 \pmod{2^r}$$
Therefore, it is correct to consider the two options $y^{(0)}_0=0$ and $y^{(1)}_0=1$ and check if either satisfies the previous equation. In particular, we have to check whether $b\equiv a^0\equiv 1\pmod{2^r}$ or $b\equiv a^1\equiv a\pmod{2^r}$. This is what the algorithm does in Lines~20 and~21. If neither equation holds, it is thus correct to return \texttt{None} as in Line~22.

Otherwise, if the correct $y_0$ has been found, the algorithm proceeds with the main loop as usual. This time, in the $(j+1)$th iteration of the main loop, with $j\geq 0$, we compute the $(j+2)$th binary digit $d_{j+1}$ of $y$. Again, we will prove that there is only one possible solution, which is the one computed by the algorithm.

To set up a proof by induction, assume that
$$c_j\equiv a^{y_j}\equiv b\pmod{p^{r+j}}$$
for some $j\geq 0$. Let $d\in\{0,1\}$ be the next binary digit of $y$. Thus, we have the two possibilities $y^{(0)}_{j+1}=y_j$ and $y^{(0)}_{j+1}=2^{j+1}+y_j$, corresponding to $d=0$ and $d=1$, respectively. We can apply \cref{l:3} with $x:=a$, since $r\geq 1$ due to Fermat's Little Theorem, to obtain
\begin{align*}
c^{(1)}_{j+1}-c^{(0)}_{j+1}
&= a^{2^{j+1}+y_j}
 - a^{y_j} \\
&= a^{2^{j+1}}
        a^{y_j}
        \left( a^{2^{j+1}} - 1 \right) \\
&\equiv (1+2m_1)(b+2m_2)h2^{r+j} \\
&\equiv h b 2^{r + j} \\
&\equiv 2^{r + j} \pmod{2^{r + j + 1}}
\end{align*}
For some integers $m_1$ and $m_2$. The last equation follows from the fact that, as before, $2\nmid hb$. Thus, flipping the $(j+2)$th binary digit of $y$ flips the $(r+j+1)$th binary digit of $c$ (as well as, possibly, some higher-order digits of $c$). We conclude that there is a unique choice for $d$, which is precisely the one computed by the algorithm in Lines~25 and 27 (note that $g=1$, which is correct).
\end{proof}

\section{Optimizations and Performance Analysis}\label{s:4}
\mypar{Measuring complexity.} We will analyze the complexity of our algorithm in terms of the number of multiplications modulo $p^k$ it performs in the worst case, as a function of the parameters $p$ and $k$.

This is justified by the fact that additions and subtractions are relatively low-cost operations in standard multitape Turing machines, requiring only $O(n)$ steps, where $n$ is the number of digits of the operands.

By contrast, there is an abundance of multiplication algorithms, ranging from the schoolbook long multiplication algorithm in $O(n^2)$ steps to the Harvey--Hoeven algorithm in $O(n\log n)$ steps. Asymptotically faster multiplication algorithms usually introduce greater constant factors, making them advantageous only for computation with very large numbers. On the other hand, the modular exponentiation $a^j\modulo m$ is typically performed by repeated squaring, which takes roughly $2\lceil \log_2 j\rceil$ multiplications modulo $m$.

Moreover, we will assume that large integers are internally stored in base $p$ (or a power of $p$). This makes divisions and multiplications by powers of $p$ essentially free operations, which is particularly convenient for our algorithm. Indeed, the only integer division or modulo operation by a non-power of $p$ found in our algorithm is in Line~3 of \cref{alg:1}: a division by the relatively small number $p-1$, which is a non-essential normalization step and can optionally be removed without compromising the algorithm's correctness.

For the above reasons, we will measure the complexity of our algorithm by counting the number of multiplications modulo $p^k$ it performs in the worst case. As we will see, after several optimizations, this number will turn out to be $k(\lceil \log_2 p\rceil +2)+O(\log p)$.

\medskip
\mypar{Performance of Bach's algorithm.} We will describe the discrete logarithm lifting technique by Bach~\cite{Bach84} using the notation of \cref{s:2}. That is, we are given $a^z\equiv b\pmod p$ with $p\nmid a$ and seek $x$ such that $a^x\equiv b\pmod{p^k}.$ For $p\neq 2$, Bach's algorithm first computes $\theta(a)$ and $\theta(b)$, where
\begin{equation}\label{e:theta}
\theta(x)=\frac{x^{(p-1)p^{k-1}}-1}{p^k}\mod p^{k-1}
\end{equation}
Then the algorithm solves the linear equation $\theta(a)y\equiv \theta(b)\pmod{p^{k-1}}$ and, if $y$ exists, it returns
$$x=\left((z-y)p^{k-1}+y\right)\mod {p^{(p-1)p^{k-1}}}$$
The numerator of the fraction in \cref{e:theta} has to be computed modulo $p^{2k-1}$. Once $\theta(a)$ and $\theta(b)$ are known, the rest of the algorithm is relatively simple, and its running time negligible.

Thus, the bottleneck of Bach's algorithm is computing the function $\theta$ twice, which essentially involves two exponentiations modulo $p^{2k-1}$, each of which takes $$2\lceil \log_2 (p-1)p^{k-1}\rceil \approx 2k\log_2p$$
multiplications modulo $p^{2k-1}$. The complexity of Bach's algorithm is therefore roughly
\begin{equation}\label{e:bach}
2\cdot 2k\log_2p\cdot M\left(\log_pp^{2k-1}\right) = 4k\log_2p\cdot M(2k-1)
\end{equation}
where $M(n)$ is the running time of the chosen multiplication algorithm when applied to operands of $n$ digits in base $p$.

We remark that this is not only the worst-case running time, but the running time in all cases, because Bach's algorithm starts out by computing $\theta(a)$ and $\theta(b)$ and then decides whether the problem has a solution or not based on the linear equation in $y$, which involves $\theta(a)$ and $\theta(b)$.

\medskip
\mypar{Performance of \cref{alg:1}.} We now study the performance of our algorithm as defined in \cref{alg:1}. We will later develop some optimizations that further reduce its running time by a constant factor.

Up to the main loop of Line~26, the only non-negligible operations are the three exponentiations modulo $p^k$ (or modulo $p$) with exponent at most $p$ required to compute $c_0$, $e_0$, and $g$. Indeed, recall that if large numbers are represented in base $p$, then computing a power of $p$ is a free operation, and therefore computing $m$ and $v_0$ takes a negligible time. This whole part takes $O(\log p \cdot M(k))$.

As for the main loop, we do at most $k$ iterations, where the only operations that do not involve multiplications or divisions by powers of $p$ are those used to compute $c_j$, $d_j$, $e_j$, and $f_j$. However, $d_j$ can be computed via multiplications modulo $p$, which are negligible. On the other hand, computing $c_j$ only takes one multiplication modulo $p^k$, while $e_j$ and $f_j$ require a more detailed analysis.

Lines~33 and~34 perform exponentiations modulo $p^k$ with exponents bounded by $p$. Thus, they take at most $4\lceil \log_2p\rceil\cdot M(k)$. These two lines are executed at most $\lfloor k/2\rfloor$ times, because in the other cases there is an optimized version in Lines~29--31.

We did not justify this part in \cref{s:2}, so here is a discussion. Let us observe that the variable $r$ is incremented at each iteration in Line~38. So, we can define $r_0=r$ as computed in Lines~15--18, and $r_{j}=r_{j-1}+1$, which is computed at the $j$th iteration. Thus, $r_j=r+j$ for $j\geq 0$.

We know that $e_j\equiv a^{(p-1)p^j}\pmod{p^k}$. Moreover, in the proof of \cref{t:1} we determined that $e_j\equiv 1\pmod{p^{r+j}=p^{r_j}}$. Let us write $e_j=1+np^{r_j}$. Now, if $r_j\geq k/2$,
\begin{equation}\label{e:opt}
f_{j+1}\equiv e_j^d\equiv \left(1+np^{r_j}\right)^d\equiv \sum_{i=0}^d{d\choose i}\left(np^{r_j}\right)^i\equiv 1+dnp^{r_j}\equiv 1 + (e_j-1)d\pmod{p^k}
\end{equation}
Accordingly, the algorithm checks if $2r_j\geq k$ in Line~29, and if so it updates $f$ as above, which takes a single multiplication modulo $p^k$. As for $e$, we have $e_{j+1}\equiv 1+(e_j-1)p\pmod{p^k}$, which is a free operation in base $p$.

The last step is the computation of $x$ in Line~39, which involves one multiplication. Putting everything together, we conclude that the total running time of the algorithm in \cref{alg:1} is at most
$$\left(O(\log p) + k + 4\lceil \log_2 p\rceil\cdot k/2+k/2)\right)M(k)=\left(k(2\lceil \log_2 p\rceil + 3/2)+O(\log p)\right)M(k)$$
Of course, the running time can be much lower if the solution $x$ is small or if $x$ does not exist, because in these cases the algorithm is able to exit the main loop early or skip it entirely.

\medskip
\mypar{Further optimizations.} Although our algorithm in \cref{alg:1} is already an improvement on Bach's algorithm, we can further optimize it to cut its running time roughly in half.

As we already observed, the bottleneck of our algorithm is updating $f$ and $e$. We managed to simplify this part when $2r_j\geq k$ thanks to \cref{e:opt}. A similar idea can be applied when $3r_j\geq k$:
$$f_{j+1}\equiv \sum_{i=0}^d{d\choose i}\left(np^{r_j}\right)^i\equiv 1+dnp^{r_j}+{d\choose 2}\left(np^{r_j}\right)^2\equiv 1 + (e_j-1)\left(d +(e_j-1)\frac{d(d-1)}{2}\right) \pmod{p^k}$$
As $d$ is a single-digit number in base $p$, this operation involves essentially two multiplications modulo $p^k$. Again, $e_{j+1}$ can be computed by the same formula with $d=p$, which is a negligible operation because multiplications by $p$ are free.

This optimization is applied when $3r_j\geq k\geq 2r_j$, as shown in Lines~32--34 of \cref{alg:3}. Hence it involves at most $k/6$ iterations of the main loop.

In the $k/3$ (or fewer) cases where $3r_j<k$, neither of these optimizations apply. One could proceed in the same fashion and extract more terms from Newton's binomial formula; this is convenient when $k$ is very large. However, in the following we will focus on optimizing the computation of $f_{j+1}\equiv e_j^d\pmod{p^k}$ and $e_{j+1}\equiv e_j^p\pmod{p^k}$ in the $k/3$ cases when Newton's binomial formula is not applied.

We can exploit the fact that both exponentiations have the same base $e_j$. Thus, in the exponentiation by squaring algorithm we can save roughly $\left\lceil \log_2 p\right\rceil$ multiplications modulo $p^k$, as shown in \cref{alg:2}. This new function performs two modular exponentiations with the same base $a$, and it takes only $3\left\lceil \log_2 b_2\right\rceil$ multiplications modulo $m$, where $b_2$ is the larger of the two exponents.

Our implementation in \cref{alg:2} falls back to a pair of exponentiations when the exponents are 4-bit numbers (Lines~2--5), because the standard \texttt{pow} function performs better on small exponents.

This optimized dual exponentiation function is invoked in Line~36 of \cref{alg:3}, which replaces Lines~33 and~34 of \cref{alg:1}. This operation takes $3\left\lceil \log_2 p\right\rceil$ multiplications modulo $p^k$.

\begin{lstlisting}[style=mylistingstyle,caption={An optimized function that computes $a^{b_1}$ and $a^{b_2}$ modulo $m$, with $b_1\leq b_2$. The threshold of 15 was determined experimentally.\label{alg:2}},captionpos=t,float,belowcaptionskip=\bigskipamount,mathescape=true]
def pow_dual(a, b1, b2, m, threshold=15):
    if b2 <= threshold:
        r1 = pow(a, b1, m)
        r2 = pow(a, b2 - b1, m)
        return r1, (r1 * r2) % m
    r1 = 1
    r2 = 1
    while b2 > 0:
        if b1 % 2 == 1:
            r1 = (r1 * a) % m
        if b2 % 2 == 1:
            r2 = (r2 * a) % m
        a = (a * a) % m
        b1 = b1 // 2
        b2 = b2 // 2
    return r1, r2
\end{lstlisting}

\begin{lstlisting}[style=mylistingstyle,caption={Optimized main loop. This block should replace Lines~26--39 of \cref{alg:1}.\label{alg:3}},captionpos=t,float,belowcaptionskip=\bigskipamount,mathescape=true,firstnumber=26]
    while b != c:
        d = (g * ((b - c) // v)) % p
        y = y + d * u
        if r + r >= k:
            f = (1 + (e - 1) * d) % m
            e = (1 + (e - 1) * p) % m
        elif r + r + r >= k:
            f = (1 + (e - 1) * (d + (e - 1) * (d * (d - 1) // 2))) % m
            e = (1 + (e - 1) * (p + (e - 1) * (p * (p - 1) // 2))) % m
        else:
            f, e = pow_dual(e, d, p, m)
        c = (c * f) % m
        u = u * p
        v = v * p
        r = r + 1
    return (p - 1) * y + z
\end{lstlisting}

\medskip
\mypar{Optimized performance.} Summarizing our analysis, we can now compute the final running time of our algorithm with the optimized main loop from \cref{alg:3}:
\begin{equation}\label{e:new2}
\left(O(\log p) + k + 3\lceil \log_2 p\rceil\cdot k/3+k/2 + 2k/6)\right)M(k)\leq \left(k(\lceil \log_2 p\rceil + 2)+O(\log p)\right)M(k)
\end{equation}
Comparing the leading term $k M(k)\log_2p$ of \cref{e:new2} with the leading term $4kM(2k-1)\log_2p$ of Bach's algorithm in \cref{e:bach}, we see that Bach's algorithm takes roughly
$$\frac{4M(2k-1)}{M(k)}$$
times longer than our optimized algorithm. Since $M(n)$ is a super-linear function (the best-known running time is $M(n)=O(n\log n)$), Bach's algorithm should take at least 8 times longer than ours to lift discrete logarithms modulo large prime powers.

In fact, extensive experimental evaluations with $p<10^3$ and $k<10^4$ have confirmed that our algorithm is at least 8 times faster than Bach's, and often much faster when $x$ is small or when no solution exists, as the main loop exits early in these cases. For larger inputs, our algorithm is expected to perform even better compared to Bach's, due to the non-linearity of $M(n)$.

\section{Conclusion}\label{s:5}
We have proposed an efficient discrete logarithm lifting algorithm, i.e., an algorithm for solving the modular equation $a^x\equiv b\pmod{p^k}$ when an integer $z$ such that $a^z\equiv b\pmod p$ is given, and $a$ is not a multiple of the prime $p$. Of course the equation has either no solutions or infinitely many solutions, since $x$ being a solution implies that $x+n\ord_{p^k}(a)$ is also a solution for any $n\in\mathbb Z$.

When a solution exists, our algorithm expresses it as $x=(p-1)y+z$, where $y\geq 0$ is smallest possible. This number $x$ is guaranteed to be the minimum non-negative solution to the modular equation, provided that $\ord_p(a)=p-1$. Otherwise, if $\ord_p(a)$ is known, we can use it in lieu of $p-1$ in Lines~3, 13, and~39 of \cref{alg:1} and in Line~41 of \cref{alg:3}. As a result, the returned value $x$ satisfies $0\leq x<\ord_{p^k}(a)$, and is therefore the minimum non-negative solution (this easily follows from our analysis of the algorithm in \cref{t:1}).

If large integers are represented modulo $p$ (or powers of $p$) our algorithm only requires
$$k(\lceil \log_2 p\rceil +2)+O(\log p)$$
multiplications modulo $p^k$ in the worst case (typically much fewer operations are needed, especially when the solution $x$ is small or does not exist). This is an improvement by at least a factor of 8 on a previous algorithm by Bach~\cite{Bach84}.

The bottleneck of our algorithm is, by far, the call to \verb|pow_dual| in Line~36 of \cref{alg:3}. Future optimization efforts should focus on improving this step.

\subsection*{Acknowledgments}
The authors would like to thank Kouichi Sakurai and Francesco Veneziano for helpful discussions and suggestions.

\bibliographystyle{plainurl}
\bibliography{bibliography}
\end{document}